\documentclass[a4paper,english,fontsize=11pt,parskip=half,abstract=true]{scrartcl}
\usepackage{babel}
\usepackage[utf8]{inputenc}
\usepackage[T1]{fontenc}
\usepackage[a4paper,left=20mm,right=20mm,top=30mm,bottom=30mm]{geometry}
\usepackage{amsmath}
\usepackage{amsthm}
\usepackage{amssymb}
\usepackage{enumerate}
\usepackage{thmtools}
\usepackage{booktabs}
\usepackage{tikz}
\usepackage[bookmarks=true,
            pdftitle={Blocks with small-dimensional basic algebra},
            pdfauthor={Benjamin Sambale},
            pdfkeywords={blocks, basic algebras, Donovan's conjecture},
            pdfstartview={FitH}]{hyperref}

\newtheorem{Thm}{Theorem} 
\newtheorem{Lem}[Thm]{Lemma}
\newtheorem{Prop}[Thm]{Proposition}
\numberwithin{equation}{section}

\newenvironment{smallmatrix2}{\bigl(\begin{smallmatrix}}{\end{smallmatrix}\bigr)}

\allowdisplaybreaks[1]

\renewcommand{\phi}{\varphi}
\newcommand{\C}{\mathrm{C}}
\newcommand{\N}{\mathrm{N}}
\newcommand{\Z}{\mathrm{Z}}
\newcommand{\ZZ}{\mathbb{Z}}

\newcommand{\Out}{\operatorname{Out}}

\newcommand{\GL}{\operatorname{GL}}
\newcommand{\SL}{\operatorname{SL}}
\newcommand{\GU}{\operatorname{GU}}
\newcommand{\GammaL}{\Gamma\mathrm{L}}
\newcommand{\PSU}{\operatorname{PSU}}
\newcommand{\PSL}{\operatorname{PSL}}
\newcommand{\PGL}{\operatorname{PGL}}

\newcommand{\Irr}{\operatorname{Irr}}
\newcommand{\IBr}{\operatorname{IBr}}

\mathchardef\ordinarycolon\mathcode`\:  
 \mathcode`\:=\string"8000
 \begingroup \catcode`\:=\active
   \gdef:{\mathrel{\mathop\ordinarycolon}}
 \endgroup

\title{Blocks with small-dimensional basic algebra}
\author{Benjamin Sambale\footnote{Institut für Algebra, Zahlentheorie und Diskrete Mathematik, Leibniz Universität Hannover, Welfengarten 1, 30167 Hannover, Germany,
\href{mailto:sambale@math.uni-hannover.de}{sambale@math.uni-hannover.de}}}
\date{\today}

\begin{document}
\frenchspacing
\maketitle
\begin{abstract}\noindent
Linckelmann and Murphy have classified the Morita equivalence classes of $p$-blocks of finite groups whose basic algebra has dimension at most $12$. We extend their classification to dimension $13$ and $14$. As predicted by Donovan's Conjecture, we obtain only finitely many such Morita equivalence classes.
\end{abstract}


\section{Introduction}
Let $F$ be an algebraically closed field of characteristic $p>0$. Donovan's Conjecture (over $F$) states that for every finite $p$-group $D$ there are only finitely many Morita equivalence classes of $p$-blocks of finite groups with defect group $D$. Since a general proof seems illusive at present, mathematicians have focused on certain families of $p$-groups $D$. 
This has culminated in a proof of Donovan's Conjecture for all abelian $2$-groups by Eaton--Livesey~\cite{ELDC}. A different approach, introduced by Linckelmann~\cite{LinckelmannBasic}, aims to classify blocks $B$ with a given \emph{basic algebra} $A$. Recall that $A$ is the unique $F$-algebra (up to isomorphism) of smallest dimension which is Morita equivalent to $B$. Linckelmann and Murphy~\cite{LinckelmannBasic,LM} have classified all blocks $B$ such that $\dim A\le 12$. Since the order of a defect group is bounded in terms of $\dim A$ (see next section), one expects only finitely many such blocks up to Morita equivalence. Indeed the list in \cite{LinckelmannBasic} is finite.
We extend their classification as follows.

\begin{Thm}\label{main}
Let $B$ be a block of a finite group with basic algebra $A$.
\begin{enumerate}[(I)]
\item If $\dim A=13$, then $B$ is Morita equivalent to one of the following block algebras:
\begin{enumerate}[(a)]
\item $FC_{13}$ ($p=13$).
\item the principal $13$-block of $\PSL(3,3)$ with defect $1$.
\item the principal $17$-block of $\PSL(2,16)$ with defect $1$.
\item the principal $2$-block of $\PGL(2,7)$ with defect group $D_{16}$.
\item a non-principal $2$-block of $3.M_{10}$ with defect group $SD_{16}$. 
\item a non-principal $7$-block of $6.A_7$ with defect $1$.
\end{enumerate}
\item If $\dim A=14$, then $B$ is Morita equivalent to one of the following block algebras:
\begin{enumerate}[(a)]
\item $FD_{14}$ ($p=7$).
\item the principal $5$-block of $S_5$ with defect $1$.
\item the principal $7$-block of $\PSU(3,3)$ with defect $1$.
\item the principal $19$-block of $\PSL(2,37)$ with defect $1$.
\end{enumerate}
\end{enumerate}
\end{Thm}

The bulk of the proof is devoted to the non-existence of a certain block with extraspecial defect group of order $27$ and exponent $3$. The methods are quite different from those in \cite{LM}.
For some of the Brauer tree algebras occurring in \cite{LinckelmannBasic} no concrete block algebra was given. For future reference we provide explicit examples in the following table. Here, $B_0$ and $B_1$ denote the principal block and a suitable non-principal block respectively.
 
\[\begin{array}{ccc}
\dim(A)&D&\text{Morita classes}\\\toprule
\le5&|D|=\dim(A)&FD\\
6&C_3&FS_3\\
7&C_5&B_0(A_5)\\
&C_7&FC_7\\
8&C_7&B_0(\PSL(2,13))\\
&|D|=8&FD\\
9&C_9&FC_9,B_0(\PSL(2,8))\\
&C_3\times C_3&F[C_3\times C_3],B_1(2.(S_3\times S_3))\\
10&C_5&FD_{10}\\
&C_{11}&B_0(\PSL(2,32))\\
11&C_7&B_0(\PSL(2,7))\\
&D_8&FS_4\\
&C_{11}&FC_{11}\\
&C_{13}&B_0(\PSL(2,25))\\
12&C_2\times C_2&FA_4\\
13&C_7&B_1(6.A_7)\\
&C_{13}&FC_{13},B_0(\PSL(3,3))\\
&D_{16}&B_0(\PGL(2,7))\\
&SD_{16}&B_1(3.M_{10})\\
&C_{17}&B_0(\PSL(2,16))\\
14&C_5&B_0(S_5)\\
&C_7&FD_{14},B_0(\PSU(3,3))\\
&C_{19}&B_0(\PSL(2,37))
\end{array}\]

For basic algebras of dimension $15$ there are still only finitely many corresponding Morita equivalences classes of blocks, but we do not know if a certain Brauer tree algebra actually occurs as a block. The details are described in the last section of this paper.

\section{Preliminaries}

Before we start the proof of \autoref{main}, we introduce a number of tools some of which were already applied in \cite{LinckelmannBasic}. For more detailed definitions we refer the reader to \cite{habil}.

Probably the most important Morita invariant of a block $B$ is the \emph{Cartan matrix} $C$. It is a non-negative, integral, symmetric, positive definite and indecomposable matrix of size $l(B)\times l(B)$ where $l(B)$ denotes the number of simple modules of $B$. Since the simple modules of a basic algebra are $1$-dimensional, the sum of the entries of $C$ equals $\dim A$ in the situation of \autoref{main}.
The largest elementary divisor of $C$ is the order of a defect group $D$ of $B$ and therefore a power of $p$. In particular, $|D|$ is bounded in terms of $\dim A$. Another Morita invariant is the isomorphism type of the \emph{center} $\Z(B)$ of $B$. In particular, 
\[k(B):=\dim\Z(B)=\dim\Z(A)\le\dim A\] 
in the situation of \autoref{main}. 

Since we encounter many blocks of defect $1$ in the sequel, it seems reasonable to construct them first.

\begin{Prop}\label{tree}
Let $B$ be a $p$-block of a finite group with defect $1$ and basic algebra $A$. Then $m:=(p-1)/l(B)$ is an integer, called the \emph{multiplicity} of $B$. If $l(B)=1$, then $\dim A=p$. If $l(B)=2$, then $\dim A\in\{2p,m+5\}$. If $l(B)=3$, then $\dim A\in\{3p,m+9,m+11,4m+6\}$. Moreover, if $\dim A\in\{13,14\}$, then only the blocks in \autoref{main} occur up to Morita equivalence.
\end{Prop}
\begin{proof}
By the Brauer--Dade theory, $B$ is determined up to Morita equivalence by a planarly embedded \emph{Brauer tree}, the multiplicity $m$ and the position of the so-called \emph{exceptional vertex} if $m>1$. For precise definitions we refer to \cite[Chapter~11]{Navarro}. 
If $l(B)=1$, then $B$ has Cartan matrix $(p)$ and the result follows (the Brauer tree has only two vertices). Now we construct the Brauer trees and Cartan matrices for $l(B)\in\{2,3\}$. The exceptional vertex is depicted by the black dot (if $m>1$).
\begin{enumerate}[(i)]
\item 
\begin{align*}
\begin{tikzpicture}[thick]
\node[draw,circle,scale=.5] (a) {};
\node[draw,right of=a,circle,fill,scale=.5] (b) {};
\node[draw,right of=b,circle,scale=.5] (c) {};
\draw (a)--(b)--(c);
\end{tikzpicture}&&
C=\begin{pmatrix}
m+1&m\\m&m+1
\end{pmatrix}&&
\dim A=4m+2=2p.
\end{align*} 
This case occurs for $B=FD_{2p}=A$. If $p=7$, we get $\dim A=14$.

\item
\begin{align*}
\begin{tikzpicture}[thick]
\node[draw,circle,fill,scale=.5] (a) {};
\node[draw,right of=a,circle,scale=.5] (b) {};
\node[draw,right of=b,circle,scale=.5] (c) {};
\draw (a)--(b)--(c);
\end{tikzpicture}&&
C=\begin{pmatrix}
m+1&1\\1&2
\end{pmatrix}&&
\dim A=m+5=\frac{p+9}{2}.
\end{align*}
This case occurs for the principal block of $\PSL(2,q)$ whenever $p$ divides $q+1$ exactly once (see \cite[Section~8.4.3]{Bonnafe}). By Dirichlet's Theorem there always exists a prime $q\equiv-1+p\pmod{p^2}$ which does the job. Choosing $(p,q)\in\{(17,16),(19,37)\}$ yields blocks with $\dim A=13$ and $\dim A=14$ respectively.

\item 
\begin{align*}
\begin{tikzpicture}[thick,baseline=(a.center)]
\node[draw,circle,scale=.5] (a) {};
\node[draw,right of=a,circle,fill,scale=.5] (b) {};
\node[draw,above right of=b,circle,scale=.5] (c) {};
\node[draw,below right of=b,circle,scale=.5] (d) {};
\draw (a)--(b)--(c)--(b)--(d);
\end{tikzpicture}&&
C=\begin{pmatrix}
m+1&m&m\\m&m+1&m\\m&m&m+1
\end{pmatrix}&&
\dim A=9m+3=3p.
\end{align*}
This case occurs for $B=F[C_p\rtimes C_3]=A$. Obviously, there are no such blocks with $\dim A\in\{13,14\}$.

\item 
\begin{align*}
\begin{tikzpicture}[thick,baseline=(a.center)]
\node[draw,circle,fill,scale=.5] (a) {};
\node[draw,right of=a,circle,scale=.5] (b) {};
\node[draw,above right of=b,circle,scale=.5] (c) {};
\node[draw,below right of=b,circle,scale=.5] (d) {};
\draw (a)--(b)--(c)--(b)--(d);
\end{tikzpicture}&&
C=\begin{pmatrix}
m+1&1&1\\1&2&1\\1&1&2
\end{pmatrix}&&
\dim A=m+11=\frac{p+32}{3}.
\end{align*}
We do not know if this tree always occurs as a block algebra, but it does for a non-principal $7$-block of the $6$-fold cover $6.A_7$ (see \cite{MOC2}). This gives an example with $\dim A=13$. Obviously, $\dim A=14$ cannot occur here.

\item 
\begin{align*}
\begin{tikzpicture}[thick]
\node[draw,circle,fill,scale=.5] (a) {};
\node[draw,right of=a,circle,scale=.5] (b) {};
\node[draw,right of=b,circle,scale=.5] (c) {};
\node[draw,right of=c,circle,scale=.5] (d) {};
\draw (a)--(b)--(c)--(d);
\end{tikzpicture}&&
C=\begin{pmatrix}
m+1&1&0\\1&2&1\\0&1&2
\end{pmatrix}&&
\dim A=m+9=\frac{p+26}{3}.
\end{align*}
By \cite[Proposition~2.1]{Naehrig}, there exists a prime $q$ such that $p$ divides $q^3-1$ exactly once. Then the principal block of $\GL(3,q)$ has this form by Fong--Srinivasan~\cite{FStrees}. The principal $13$-block of $\PSL(3,3)$ is an example with $\dim A=13$. Again, $\dim A=14$ is impossible here.

\item 
\begin{align*}
\begin{tikzpicture}[thick]
\node[draw,circle,scale=.5] (a) {};
\node[draw,right of=a,circle,fill,scale=.5] (b) {};
\node[draw,right of=b,circle,scale=.5] (c) {};
\node[draw,right of=c,circle,scale=.5] (d) {};
\draw (a)--(b)--(c)--(d);
\end{tikzpicture}&&
C=\begin{pmatrix}
m+1&m&0\\m&m+1&1\\0&1&2
\end{pmatrix}&&
\dim A=4m+6=\frac{4p+14}{3}.
\end{align*}
Again by \cite[Theorem~1]{Naehrig}, there exists a prime $q$ such that the principal block of $\GU(3,q)$ has this form.
The principal $7$-block of $\PSU(3,3)$ is an example with $\dim A=14$. On the other hand, $\dim A=13$ cannot occur.
\end{enumerate}
Finally, if $l(B)\ge 4$, then the trace of $C$ is $\ge 8$ and we need at least six positive off-diagonal entries to ensure that $C$ is symmetric and indecomposable. Hence, $\dim A\le 14$ can only occur if $l(B)=4$, $\dim A=14$, $m=1$ and the Brauer tree is a line. This happens for the principal $5$-block of $S_5$. 
\end{proof}

In order to investigate blocks of larger defect, we develop some more advanced methods.
The \emph{decomposition matrix} $Q=Q_1$ of $B$ is non-negative, integral and indecomposable of size $k(B)\times l(B)$ such that $Q^\text{t}Q=C$. Given $\dim A$, there are only finitely many choices for $Q$.
Richard Brauer has introduced the so-called \emph{contribution matrix} 
\[M=M^1:=|D|QC^{-1}Q^\mathrm{t}\in\ZZ^{k(B)\times k(B)}.\] 
The \emph{heights} of the irreducible characters of $B$ are encoded in the $p$-adic valuation of $M$ (see \cite[Proposition~1.36]{habil}). As usual, we denote the number of irreducible characters of $B$ of height $h\ge 0$ by $k_h(B)$. If $k_0(B)<k(B)$, then $D$ is non-abelian according to Kessar--Malle's~\cite{KessarMalle} solution of one half of Brauer's height zero conjecture. 

The $2$-blocks occurring in \autoref{main} are determined by the following proposition.

\begin{Prop}\label{16}
Let $B$ be a block of a finite group with Cartan matrix $C=\begin{smallmatrix2}
5&2\\2&4
\end{smallmatrix2}$. Then $B$ is Morita equivalent to the principal $2$-block of $\PGL(2,7)$ or to a non-principal block of $3.M_{10}$. Moreover, there is no block with Cartan matrix
\begin{align*}
\begin{pmatrix}
5&1&1\\
1&2&0\\
1&0&2
\end{pmatrix}&&\text{or}&&
\begin{pmatrix}
6&1&0\\
1&2&1\\
0&1&2
\end{pmatrix}.
\end{align*}
\end{Prop}
\begin{proof}
All three matrices have largest elementary divisor $16$. Therefore, $p=2$ and a defect group $D$ of $B$ has order $16$.
For the first matrix, the possible decomposition matrices are
\begin{align*}
\begin{pmatrix}
1&1\\
1&1\\
.&1\\
.&1\\
1&.\\
1&.\\
1&.
\end{pmatrix},&&
\begin{pmatrix}
2&1\\
.&1\\
.&1\\
.&1\\
1&.
\end{pmatrix}.
\end{align*}
The diagonal of the contribution matrix $M^1$ is $(5,5,5,5,4,4,4)$ or $(13,5,5,5,4)$. It follows that $k_0(B)=4$ (the first four characters have height $0$). By \cite[Theorem~13.6]{habil}, the Alperin--McKay Conjecture holds for all $2$-blocks of defect $4$. Thus, $k_0(B_D)=4$ where $B_D$ is the Brauer correspondent of $B$ in $\N_G(D)$. Now $B_D$ dominates a block $\overline{B_D}$ of $\N_G(D)/D'$ with abelian defect group $D/D'$. By \cite[Theorem~9.23]{Navarro}, we conclude that \[k(\overline{B_D})=k_0(\overline{B_D})\le k_0(B_D)=4.\] 
Now \cite[Proposition~1.31]{habil} implies $|D/D'|=4$. Hence, $D$ is a dihedral group, a semidihedral group or a quaternion group.
A look at \cite[Theorem~8.1]{habil} (the Cartan matrices in (5a) and (5b) are mixed up) tells us that $k(B)=7$ and $D\in\{D_{16},SD_{16}\}$. 
The corresponding Morita equivalence classes were computed by Erdmann~\cite{Erdmann} (see \cite[Appendix]{HolmHabil} for a definite list). Only the two stated examples occur up to Morita equivalence.

For the second matrix there is only one possible decomposition matrix and we obtain similarly that $k_0(B)=4$ and $k(B)=7$. By \cite[Theorem~8.1]{habil}, $D\cong D_{16}$. However, it can be seen from \cite[Appendix]{HolmHabil} that there are no such blocks (all Cartan invariants are positive). Nevertheless, $C$ occurs as Cartan matrix with respect to a suitable basic set (for the principal block of $\PSL(2,17)$, for instance).

In the last case there are two feasible decomposition matrices:
\begin{align*}
\begin{pmatrix}
2&.&.\\
1&1&.\\
1&.&.\\
.&1&1\\
.&.&1
\end{pmatrix},&&
\begin{pmatrix}
1&1&.\\
1&.&.\\
1&.&.\\
1&.&.\\
1&.&.\\
1&.&.\\
.&1&1\\
.&.&1
\end{pmatrix}.
\end{align*}
The first matrix leads to $k_0(B)=4$ and $k(B)=5$. This contradicts \cite[Theorem~8.1]{habil}.
The second matrix reveals $k_0(B)=k(B)=8$. Since Brauer's height zero conjecture holds for $B$ by \cite[Theorem~13.6]{habil}, $D$ is abelian. By \cite[Theorem~8.3]{habil}, $D$ is not isomorphic to $C_4\times C_4$. 
In fact, $D$ must be elementary abelian by \cite[Proposition~16]{SambaleC4}, for instance. By Eaton's classification~\cite{EatonE16}, $B$ should be Morita equivalent to the group algebra of the Frobenius group $D\rtimes C_3$. But this is a basic algebra of dimension $48$. 
\end{proof} 

The local structure of $B$ is determined by a \emph{fusion system} $\mathcal{F}$ on $D$ (again there are only finitely many choices for $\mathcal{F}$ when $\dim A$ is fixed). The $p'$-group $E:=\Out_{\mathcal{F}}(D)$ is called the \emph{inertial quotient} of $B$. Recall that for every $S\le D$ there is exactly one \emph{subpair} $(S,b_S)$ attached to $\mathcal{F}$ (here, $b_S$ is a Brauer correspondent of $B$ in $\C_G(S)$). After $\mathcal{F}$-conjugation, we may and will always assume that $S$ is \emph{fully $\mathcal{F}$-normalized}. Then $b_S$ has defect group $\C_D(S)$ and fusion system $\C_{\mathcal{F}}(S)$. Moreover, the Brauer correspondent $B_S:=b_S^{\N_G(S,b_S)}$ has defect group $\N_D(S)$ and fusion system $\N_{\mathcal{F}}(S)$. If $S=\langle u\rangle$ is cyclic, we call $(u,b_u):=(S,b_S)$ a \emph{subsection}.

Let $\mathcal{R}$ be a set of representatives for the $\mathcal{F}$-conjugacy classes of elements in $D$. 
Then a formula of Brauer asserts that 
\[k(B)=\sum_{u\in\mathcal{R}}l(b_u).\] 
Each $b_u$ dominates a block $\overline{b_u}$ of $\C_G(u)/\langle u\rangle$ with defect group $\C_D(u)/\langle u\rangle$ and fusion system $\C_{\mathcal{F}}(u)/\langle u\rangle$. If $\overline{C_u}$ is the Cartan matrix of $\overline{b_u}$, then $C_u:=|\langle u\rangle|\overline{C_u}$ is the Cartan matrix of $b_u$. Let $Q_u:=(d_{\chi\phi}^u:\chi\in\Irr(B),\phi\in\IBr(b_u))$ be the \emph{generalized decomposition matrix} with respect to $(u,b_u)$. The orthogonality relations assert that $Q_u^\text{t}\overline{Q_v}=\delta_{uv}C_u$ for $u,v\in\mathcal{R}$ where $\delta_{uv}$ is the Kronecker delta and $\overline{Q_v}$ is the complex conjugate of $Q_v$. As above, we define the contribution matrices $M^u$ for each $u\in\mathcal{R}$. 
Since the generalized decomposition numbers are algebraic integers, we may express $Q_u$ with respect to a suitable integral basis. This yields “fake” decomposition matrices $\widetilde{Q}_u$ which obey similar orthogonality relations (see \cite[Section~4]{SambaleC5} for details). We call $\widetilde{C}_u:=\widetilde{Q}_u^\text{t}\widetilde{Q}_u$ the “fake” Cartan matrix of $b_u$. 

The following curious result might be of independent interest.

\begin{Prop}\label{cur}
Let $B$ be a $p$-block of a finite group with abelian defect group $D$ and inertial quotient $E$. 
\begin{enumerate}[(i)]
\item If $p=2$, then $l(B)\equiv|E|\equiv k(E)\pmod{8}$.
\item If $p=3$, then $l(B)\equiv|E|\equiv k(E)\pmod{3}$.
\end{enumerate}
\end{Prop}
\begin{proof}
We argue by induction on $|D|$. If $|D|\le 4$, then $l(B)=|E|=k(E)$. Thus, let $|D|\ge 8$. Let $d:=8$ if $p=2$ and $d:=3$ if $p=3$.
Let $\mathcal{R}$ be a set of representatives for the $E$-orbits on $D$. Since $E$ is a $p'$-group, we have $|\C_E(u)|^2\equiv 1\pmod{d}$ for all $u\in D$. Hence, 
\[|E|\sum_{u\in\mathcal{R}}|\C_E(u)|=\sum_{u\in D}|\C_E(u)|^2\equiv |D|\equiv 0\pmod{d}.\]
By Kessar--Malle~\cite{KessarMalle} and \cite[Proposition~1.31]{habil}, $k(B)=k_0(B)\equiv 0\pmod{d}$.
Using Brauer's formula and induction yields
\[l(B)=k(B)-\sum_{u\in\mathcal{R}\setminus\{1\}}l(b_u)\equiv -\sum_{u\in\mathcal{R}\setminus\{1\}}|\C_E(u)|\equiv|E|\equiv\sum_{\chi\in\Irr(E)}\chi(1)^2\equiv k(E)\pmod{d}.\qedhere\]
\end{proof}

For the principal block $B$, Alperin's weight conjecture asserts that $l(B)=k(E)$ in the situation of \autoref{cur}.

Finally, we study the elementary divisors of $C$ via the theory of \emph{lower defect groups}. The \emph{$1$-multiplicity} $m^{(1)}_B(S)$ of a subgroup $S\le D$ is defined as the dimension of a certain section of $\Z(B)$ (the precise definition in \cite[Section~1.8]{habil} is not needed here). Since we are only interested in $1$-multiplicities, we omit the exponent $(1)$ from now on. Furthermore, it is desirable to attached a multiplicity to a subpair $(S,b_S)$ instead of a subgroup. We do so by setting 
\[m_B(S,b_S):=m_{B_S}(S).\] 
Note that $(S,b_S)$ is also a subpair for $B_S$ and $m_{B_S}(S,b_S)=m_B(S,b_S)$. 
Now the multiplicity of an elementary divisor $d$ of $C$ is 
\[m(d)=\sum m_B(S,b_S)\] 
where $(S,b_S)$ runs through the $\mathcal{F}$-conjugacy classes of subpairs with $|S|=d$. In particular, $m_B(D,b_D)=m(|D|)=1$.

We are now in a position to investigate blocks with extraspecial defect group $D\cong 3^{1+2}_+$ of order $27$ and exponent $3$. The partial results on these blocks obtained by Hendren~\cite{Hendren1} are not sufficient for our purpose. We proceed in four stages. The first lemma is analogous to \cite[Lemma~13.3]{habil}.

\begin{Lem}\label{lem1}
Let $B$ be a block of a finite group $G$ with defect group $D\cong C_3\times C_3$ and inertial quotient $E\cong C_2\times C_2$. Suppose that $l(B)=4$. Let $D=S\times T$ with $E$-invariant subgroups $S\cong T\cong C_3$. Then $m_B(S,b_S)=m_B(T,b_T)=1$.
\end{Lem}
\begin{proof}
By \cite[Theorem~3]{SambaleC5}, $B$ is perfectly isometric to its Brauer correspondent in $\N_G(D)$. It follows that the elementary divisors of the Cartan matrix of $B$ are $1,3,3,9$. In particular, $m(3)=2$. Let $U\le D$ be of order $3$ such that $S\ne U\ne T$. Then $b_U$ is nilpotent and $l(b_U)=1$. Since $B_U$ has defect group $D$, we obtain $m_{B_U}(D)=1$. Hence, \cite[Lemma~1.43]{habil} implies $m_B(U,b_U)=m_{B_U}(U)=0$. It follows that 
\begin{equation}\label{ldef}
m_B(S,b_S)+m_B(T,b_T)=m(3)=2. 
\end{equation}
Similarly, $b_S$ has defect group $D$ and inertial quotient $C_2$. Hence, $l(b_S)=2$ by \cite[Theorem~3]{SambaleC5}.
This time \cite[Lemma~1.43]{habil} gives
\[m_B(S,b_S)=m_{B_S}(S)+m_{B_S}(D)-1\le l(b_S)-1=1\]
and similarly, $m_B(T,b_T)\le 1$. By \eqref{ldef}, we must have equality.
\end{proof}

Recall that every $3'$-automorphism group $E$ of $D\cong 3^{1+2}_+$ acts faithfully on $D/\Phi(D)\cong C_3\times C_3$. This allows us to regard $E$ as a subgroup of the semilinear group $\GammaL(1,9)\le\GL(2,3)$. Note that $\GammaL(1,9)$ is isomorphic to the semidihedral group $SD_{16}$. Moreover, $\C_E(\Z(D))=E\cap\SL(2,3)\le Q_8$.

\begin{Lem}\label{lem2}
Let $B$ be a block of a finite group $G$ with defect group $D\cong 3^{1+2}_+$ and inertial quotient $E\cong SD_{16}$. Suppose that $Z:=\Z(D)\unlhd G$ and that $\IBr(b_Z)$ contains at least four Brauer characters which are not $G$-invariant. Then $m_B(Z,b_Z)>0$.
\end{Lem}
\begin{proof}
Since $\C_E(Z)\cong Q_8$ acts regularly on $D/Z$, there are two subgroups, say $Z$ and $S$, of order $3$ in $D$ up to $\mathcal{F}$-conjugation. Hence, $m(3)=m_B(Z,b_Z)+m_B(S,b_S)$. We observe that $B_S$ has defect group $\N_D(S)=SZ\cong C_3\times C_3$ and inertial quotient $C_2\times C_2$. By \cite[Theorem~3]{SambaleC5}, $l(B_S)\in\{1,4\}$. In the first case, $m_B(S,b_S)=0$ by \cite[Lemma~1.43]{habil} and in the second case $m_B(S,b_S)=m_{B_S}(S,b_S)=1$ by \autoref{lem1}. Thus, it suffices to show that $m(3)\ge 2$. 

Since $E$ acts non-trivially on $Z$, we have $|G:N|=2$ where $N:=\C_G(Z)$.
As usual, $b_Z$ dominates a block $\overline{b_Z}$ with defect group $D/Z\cong C_3\times C_3$ and inertial quotient $\C_E(Z)\cong Q_8$. By hypothesis, $l(b_Z)\ge 4$. 
By \cite[Lemma~13]{SambaleC5}, there exists a basic set $\Gamma$ for $\overline{b_Z}$ (which is a basic set for $b_Z$ as well) such that $G$ acts on $\Gamma$ and the Cartan matrix of $b_Z$ with respect to $\Gamma$ is 
\begin{align*}
3\begin{pmatrix}
2&1&1&1&2\\
1&2&1&1&2\\
1&1&2&1&2\\
1&1&1&2&2\\
2&2&2&2&5
\end{pmatrix}&&\text{or}&&3(1+\delta_{ij})_{i,j=1}^8.
\end{align*}
We may assume that $\theta_1,\ldots,\theta_4\in\Gamma$ such that $\phi:=\theta_1^G=\theta_2^G$ and $\mu:=\theta_3^G=\theta_4^G$ belong to a basic set $\Delta$ of $B$.  
In order to determine the Cartan matrix $C$ of $B$ with respect to $\Delta$, we introduce the projective indecomposable characters $\Phi_\phi$ and $\Phi_\mu$ (note that these are generalized characters in our setting). By \cite[Theorem~8.10]{Navarro}, $\Phi_{\phi}=\Phi_{\theta_1}^G$ and $\Phi_\mu=\Phi_{\theta_3}^G$. In particular, $\Phi_\phi$ and $\Phi_\mu$ vanish outside $N$. 
We compute
\begin{align*}
[\Phi_{\phi},\Phi_{\phi}]&=\frac{1}{|G|}\sum_{g\in G}|\Phi_\phi(g)|^2=\frac{1}{2}\frac{1}{|N|}\sum_{g\in N}|\Phi_\phi(g)|^2=\frac{1}{2}[\Phi_{\theta_1}+\Phi_{\theta_2},\Phi_{\theta_1}+\Phi_{\theta_2}]=9=[\Phi_{\mu},\Phi_{\mu}],\\
[\Phi_\phi,\Phi_\mu]&=\frac{1}{2}[\Phi_{\theta_1}+\Phi_{\theta_2},\Phi_{\theta_3}+\Phi_{\theta_4}]=6.
\end{align*}
Let $\tau\in\Delta\setminus\{\phi,\mu\}$. If $\tau_N$ is the sum of two characters in $\Gamma$, then $l(b_Z)=8$ and 
\[[\Phi_\phi,\Phi_\tau]=6=[\Phi_\mu,\Phi_\tau].\]
If, on the other hand, $\tau_N\in\Gamma$, then also $(\Phi_\tau)_N=\Phi_{\tau_N}$ by \cite[Corollary~8.8]{Navarro}. In this case we compute
\[[\Phi_\phi,\Phi_\tau]=[\Phi_\mu,\Phi_\tau]\in\{3,6\}\]
depending on $l(b_Z)$. 
In any case, $C$ has the form 
\[C=\begin{pmatrix}
9&6&a_1&\cdots&a_s\\
6&9&a_1&\cdots&a_s\\
a_1&a_1&*&\cdots&*\\
\vdots&\vdots&\vdots&&\vdots\\
a_s&a_s&*&\cdots&*
\end{pmatrix}\]
with $a_1,\ldots,a_s\in\{3,6\}$. By the Gauss algorithm there exist $X,Y\in\GL(l(B),\ZZ)$ such that
\[XCY=\begin{pmatrix}
3&.&.\\
.&3&.\\
.&.&*
\end{pmatrix}.\]
Since all elementary divisors of $C$ are powers of $3$, it follows that $m(3)\ge 2$ as desired.
\end{proof}

\begin{Lem}\label{lem3}
Let $B$ be a block of a finite group $G$ with defect group $D\cong 3^{1+2}_+$ and fusion system $\mathcal{F}=\mathcal{F}(J_4)$. Then $B$ cannot have Cartan matrix $\begin{smallmatrix2}7&1\\1&4
\end{smallmatrix2}$.
\end{Lem}
\begin{proof}
By way of contradiction, suppose that $B$ has the given Cartan matrix $C$.
Then $B$ has decomposition matrix
\begin{align*}
\begin{pmatrix}
2&.\\
1&.\\
1&.\\
.&1\\
.&1\\
.&1\\
1&1
\end{pmatrix}&&\text{or}&&
\begin{pmatrix}
1&.\\
1&.\\
1&.\\
1&.\\
1&.\\
1&.\\
.&1\\
.&1\\
.&1\\
1&1
\end{pmatrix}.
\end{align*}
The diagonal of the contribution matrix $M^1$ is $(16,4,4,7,7,7,9)$ or $(4,4,4,4,4,4,7,7,7,9)$. It follows that $k_0(B)\in\{6,9\}$ and $k_1(B)=1$ (the last row corresponds to the character of height $1$). 
From the Atlas we know that all non-trivial elements of $D$ are $\mathcal{F}$-conjugate. Let $(z,b_z)$ be a non-trivial subsection such that $z\in Z:=\Z(D)$. By \cite[Table~1.2]{ExtraspecialExpp}, $B$ has inertial quotient $SD_{16}$.
It follows that $b_z$ is a block with defect group $D$ and inertial quotient $Q_8$. Moreover, $l(b_z)=k(B)-l(B)\in\{5,8\}$.
The possible Cartan matrices of $b_z$ are given in the proof of \autoref{lem2}. 
The generalized decomposition numbers $d_{\chi\phi}^z$ are Eisenstein integers and can be expressed with respect to the integral basis $1,e^{2\pi i/3}$. 
According to the action of $\N_G(Z,b_z)$ on $\IBr(b_z)$ there are eight possibilities for the “fake” Cartan matrix $\widetilde{C}_z$ which are listed explicitly in \cite[proof of Lemma~14]{SambaleC5}. In each case we apply an algorithm of Plesken~\cite{Plesken} (implemented in GAP~\cite{GAP48}) to determine the feasible “fake” decomposition matrices $\widetilde{Q}_z$. To this end we also take into account that the diagonal of $M^z$ is $(11,23,23,20,20,20,18)$ or $(23,23,23,23,23,23,20,20,20,18)$, since $M^1+M^z=|D|1_{k(B)}$.  
It turns out that only two of the eight cases can actually occur. If $k(B)=7$, then $\N_G(Z,b_z)$ has one fixed point in $\IBr(b_z)$ and if $k(B)=10$, then $\N_G(Z,b_z)$ has two fixed points in $\IBr(b_z)$. Hence, in both cases the block $B_Z$ fulfills the assumption of \autoref{lem2}. Consequently, $m(3)= m_B(Z,b_Z)=m_{B_Z}(Z,b_Z)>0$. However, the elementary divisors of $C$ are $1$ and $27$. Contradiction.
\end{proof}

\begin{Prop}\label{27}
There does not exist a block of a finite group with Cartan matrix $\begin{smallmatrix2}7&1\\1&4
\end{smallmatrix2}$.
\end{Prop}
\begin{proof}
As in \autoref{lem3}, any block $B$ with the given Cartan matrix $C$ has a defect group $D$ of order $27$. 
The possible decomposition matrices were also computed in the proof of \autoref{lem3}. In particular, $k_0(B)\in\{6,9\}$, $k_1(B)=1$ and $k(B)-l(B)\in\{5,8\}$. By Kessar--Malle~\cite{KessarMalle}, $D$ is nonabelian. By \cite[Theorem~8.15]{habil}, $D$ cannot have exponent $9$, i.\,e. $D\cong 3^{1+2}_+$. The fusion systems $\mathcal{F}$ on that group were classified in Ruiz--Viruel~\cite{ExtraspecialExpp}. As explained before, we regard the inertial quotient $E$ of $B$ as a subgroup of $SD_{16}$.
Let $\mathcal{R}$ be a set of representatives for the $\mathcal{F}$-conjugacy classes in $D$. For $1\ne u\in\mathcal{R}$ we have $l(b_u)\equiv|\C_E(u)|\pmod{3}$ by \autoref{cur} (applied to $\overline{b_u}$ if $u\in Z:=\Z(D)$). Therefore, the residue of $k(B)-l(B)$ modulo $3$ only depends on $\mathcal{F}$. If $D$ contains $\mathcal{F}$-essential subgroups, then $\mathcal{F}$ is the fusion system of one of the following groups $H$:
\[C_3^2\rtimes\SL(2,3),\ C_3^2\rtimes\GL(2,3),\ \PSL(3,3),\ \PSL(3,3).2,\ ^2F_4(2)',\ J_4.\]
The last case was excluded in \autoref{lem3}. In the remaining cases we can compare with the principal block of $H$ to derive the contradiction
\[2\equiv k(B)-l(B)\equiv k(B_0(H))-l(B_0(H))\not\equiv2\pmod{3}.\]

Hence, there are no $\mathcal{F}$-essential subgroups, i.\,e. $\mathcal{F}=\mathcal{F}(D\rtimes E)$. Suppose that $E\le Q_8$. Then $\N_G(Z,b_Z)=\C_G(Z)$ and $b_Z=B_Z$ has fusion system $\mathcal{F}$ as well. If $E=1$, then $B$ is nilpotent in contradiction to $l(B)=2$. Thus, let $E\ne 1$. Let $\overline{B_Z}$ be the block with defect group $D/Z$ dominated by $B_Z$. By \cite[Theorem~3]{SambaleC5} and \autoref{cur}, $l(B_Z)=l(\overline{B_Z})\ge 2$. Since $E$ acts semiregularly on $D/Z$, the Cartan matrix of $\overline{B_Z}$ has elementary divisors $1$ and $9$ (see \cite[Proposition~1.46]{habil}). Hence, $3$ is an elementary divisor of the Cartan matrix of $B_Z$. Since $Z\le\Z(\C_G(Z))$, it follows that $m(3)\ge m_B(Z,b_Z)=m_{B_Z}(Z,b_Z)>0$ by \cite[Lemma~1.44]{habil}. A contradiction.

We are left with the situation $E\nsubseteq Q_8$. Here, $\mathcal{R}\cap Z=\{1,z\}$. The case $E\cong C_2\times C_2$ is impossible by a comparison with the principal block of $D\rtimes E$ as above. 
We summarize the remaining cases (the second column refers to the small groups library in GAP):

\[\begin{array}{cccc}
E&\text{realizing group}&l(b_z)&\sum_{u\notin\Z(D)} l(b_u)\\\toprule
C_2&54:5&1&2+2+1+1+1\\
C_8&216:86&4&1\\
D_8&216:87&4&2+2
\end{array}\]

In the first case we have $k_0(B)=9$ and there exists $u\in\mathcal{R}\setminus Z$ such that $u$ and $u^{-1}$ are $\mathcal{F}$-conjugate. Then $l(b_u)=1$ and the Cartan matrix of $b_u$ is $(9)$. The generalized decomposition matrix $Q_u$ is integral, since $Q_u=Q_{u^{-1}}=\overline{Q_u}$. The only choice up to signs is $Q_u=(\pm1,\ldots,\pm1,0)^\text{t}$ where the last character has height $1$. However, $Q_u$ cannot be orthogonal to the decomposition matrix of $B$ as computed in \autoref{lem3}. 
Next let $E\cong C_8$. Here $k_0(B)=6$ and the generalized decomposition matrix $Q_u$ has the form $Q_u=(\pm2,\pm1,\ldots,\pm1,0)^{\text{t}}$. More precisely, $Q_1$ and $Q_u$ can be arranged as follows
\[(Q_1,Q_u)=\begin{pmatrix}
2&.&1\\
1&.&-1\\
1&.&-1\\
.&1&2\\
.&1&-1\\
.&1&-1\\
1&1&.
\end{pmatrix}.\]
From that we compute the diagonal of the contribution matrix $M^z$ as $(8,20,20,8,17,17,18)$.
By \cite[Proposition~7]{SambaleC5} (applied to the dominated block with defect group $C_3\times C_3$), there exists a basic set $\Gamma$ for $b_z$ such that the Cartan matrix becomes $3(2+\delta_{ij})_{i,j=1}^4$
and $\N_G(Z,b_Z)$ acts on $\Gamma$. There are three such actions. In each case we may compute the “fake” Cartan matrix $\widetilde{C}_z$ and apply Plesken's algorithm. It turns out that none of those cases leads to a valid configuration. 

Finally, let $E\cong D_8$ and $u\in\mathcal{R}$ such that $l(b_u)=2$. We check that $u$ is $\mathcal{F}$-conjugate to $u^{-1}$. The Cartan matrix of $b_u$ is $3\begin{smallmatrix2}
2&1\\1&2
\end{smallmatrix2}$ up to basic sets. Let $U:=\langle u\rangle$. If $\N_G(U,b_u)$ interchanges the Brauer characters of $b_u$, then the “fake” Cartan matrix becomes $\widetilde{C}_u=\begin{smallmatrix2}
5&1\\1&2
\end{smallmatrix2}$ (see \cite[proof of Lemma~14]{SambaleC5}, for instance). But then $k_0(B)\le 6$ which is not the case. Therefore, $\N_G(U,b_u)$ fixes the Brauer characters of $b_u$ and $B_U$ satisfies $l(B_U)=4$ by Clifford theory. By \autoref{lem1}, we conclude that $m(3)\ge m_B(U,b_U)=m_{B_U}(U,b_U)=1$. This is the final contradiction.
\end{proof}

Since the contribution matrix does not depend on basic sets, the proof shows more generally that $\begin{smallmatrix2}7&1\\1&4\end{smallmatrix2}$ cannot be the Cartan matrix of a block with respect to any basic set. This is in contrast to the main result of \cite{LM} where the authors showed that $\begin{smallmatrix2}5&1\\1&2\end{smallmatrix2}$ is not the Cartan matrix of a block with defect group $C_3\times C_3$, although a transformation of basic sets results in the Cartan matrix
\[\begin{pmatrix}
5&4\\4&5
\end{pmatrix}=\begin{pmatrix}
1&0\\
1&-1
\end{pmatrix}\begin{pmatrix}
5&1\\1&2
\end{pmatrix}
\begin{pmatrix}
1&1\\0&-1
\end{pmatrix}\]
of the Frobenius group $C_3^2\rtimes C_2$. 

\section{Basic algebras of dimension 13}

Suppose that $B$ is a block with basic algebra $A$ of dimension $13$ and Cartan matrix $C$.
We discuss the various possibilities for $C$. If $l(B)=1$, then $C=(13)$, $p=13$ and $B$ has defect $1$. This is covered by \autoref{tree}.
For $l(B)=2$ we obtain the following possibilities for $C$ up to labeling of the simple modules:

\[\begin{array}{*{9}{c}}
C&\begin{pmatrix}
9&1\\1&2
\end{pmatrix}&
\begin{pmatrix}
8&1\\
1&3
\end{pmatrix}&
\begin{pmatrix}
7&1\\1&4
\end{pmatrix}&
\begin{pmatrix}
6&1\\
1&5
\end{pmatrix}&
\begin{pmatrix}
7&2\\
2&2
\end{pmatrix}&
\begin{pmatrix}
6&2\\
2&3
\end{pmatrix}&
\begin{pmatrix}
5&2\\
2&4
\end{pmatrix}&
\begin{pmatrix}
4&3\\
3&3
\end{pmatrix}\\[2mm]
\det C&17&23&27&29&10&14&16&3
\end{array}\]

The determinants $10$ and $14$ are not prime powers. If $\det C$ is a prime, then the result follows from \autoref{tree}. The remaining cases $\det C\in\{16,27\}$ were handled in \autoref{16} and \autoref{27} respectively.

Now we turn to $l(B)=3$. Up to labeling, the following possibilities may arise:

\[\begin{array}{*{7}{c}}
C&

\begin{pmatrix}
5&1&1\\
1&2&0\\
1&0&2
\end{pmatrix}&
\begin{pmatrix}
5&1&0\\
1&2&1\\
0&1&2
\end{pmatrix}&
\begin{pmatrix}
4&1&1\\
1&3&0\\
1&0&2
\end{pmatrix}&
\begin{pmatrix}
4&1&0\\
1&3&1\\
0&1&2
\end{pmatrix}&
\begin{pmatrix}
4&0&1\\
0&3&1\\
1&1&2
\end{pmatrix}\\[2mm]
\det C&16&13&19&18&17\\\midrule
C&\begin{pmatrix}
3&1&1\\
1&3&0\\
1&0&3
\end{pmatrix}&
\begin{pmatrix}
3&1&1\\
1&2&1\\
1&1&2
\end{pmatrix}&
\begin{pmatrix}
3&2&0\\
2&2&1\\
0&1&2
\end{pmatrix}&
\begin{pmatrix}
3&2&1\\
2&2&0\\
1&0&2
\end{pmatrix}\\[2mm]
\det C&21&7&1&2
\end{array}
\]
The determinants $1$, $18$ and $21$ are impossible and the prime determinants are settled by \autoref{tree}. The remaining case was done in \autoref{16}.

If $l(B)\ge 4$, then the trace of $C$ is $\ge 8$. Since $C$ is symmetric and indecomposable, we need at least six more non-zero entries. But then $\dim A\ge 8+6=14$. 

\section{Basic algebras of dimension 14}

In this section, $B$ is a block with basic algebra $A$ of dimension $14$. Since $14$ is not a prime power, $l(B)\ge 2$. In view of \autoref{tree}, we only list the possible Cartan matrices $C$ such that $\det C$ is a prime power, but not a prime:

\[\begin{array}{*{6}{c}}
C&
\begin{pmatrix}
6&1&0\\
1&2&1\\
0&1&2
\end{pmatrix}&
\begin{pmatrix}
5&1&1\\
1&3&0\\
1&0&2
\end{pmatrix}&
\begin{pmatrix}
4&2&0\\
2&2&1\\
0&1&2
\end{pmatrix}&
\begin{pmatrix}
3&1&0\\
1&3&2\\
0&2&2
\end{pmatrix}&
\begin{pmatrix}
2&1&1&1\\
1&2&0&0\\
1&0&2&0\\
1&0&0&2
\end{pmatrix}
\\[2mm]
\det C&16&25&4&4&4
\end{array}
\]
The $2$-blocks of defect $2$ were classified by Erdmann~\cite{Erdmann}. The Morita equivalence classes are represented by $FD$, $FA_4$ and $B_0(A_5)$. Only the last block did not already appear in Linckelmann's list. It is easy to check that $B_0(A_5)$ has a basic algebra of dimension $18$. The case $\det C=16$ was done in \autoref{16}.
Now let $\det C=25$ and $p=5$. Since $l(B)=3$ does not divide $p-1=4$, $D$ is elementary abelian of order $25$. The decomposition matrix is
\[
\begin{pmatrix}
1&0&0\\
1&0&0\\
1&0&0\\
1&1&0\\
1&0&1\\
0&1&0\\
0&1&0\\
0&0&1
\end{pmatrix}.\]
In particular, $k(B)-l(B)=5$.
Let $E\le\GL(2,5)$ be the inertial quotient of $B$. Every non-trivial subsection $(u,b_u)$ satisfies $l(b_u)=|\C_E(u)|$ by Brauer--Dade. In particular, $k(B)-l(B)$ only depends on the action of $E$ on $D$. An inspection of \cite[Theorem~5]{SambaleC5} shows that $k(B)-l(B)=5$ never occurs. Hence, this case is impossible as well.

\section{The next challenge}

While classifying blocks $B$ with basic algebra of dimension $15$, only the following Cartan matrices are hard to deal with:
\begin{align*}
\begin{pmatrix}
5&1&1\\
1&2&1\\
1&1&2
\end{pmatrix},&&
\begin{pmatrix}
6&0&1\\
0&3&1\\
1&1&2
\end{pmatrix}.
\end{align*}
The first matrix belongs to a Brauer tree algebra and could potentially arise from a $13$-block of defect $1$ (see \autoref{tree}). David Craven has informed me that such a block does most likely not exist (assuming the classification of finite simple groups).

The second matrix leads, once again, to a defect group $D$ of order $27$. Moreover, $k(B)=k_0(B)\in\{6,9\}$. Arguing along the lines of \autoref{27}, it can be shown with some effort that $D$ is abelian. Now the block is ruled out by \autoref{cur}.

Finally, for basic algebras of dimension $16$, a $3\times 3$ Cartan matrix with largest elementary divisor $32$ shows up. We made no attempt to say something about such blocks.

\section*{Acknowledgment}
I thank David Craven for providing detailed information on the possible trees of block algebras.
This work is supported by the German Research Foundation (\mbox{SA 2864/1-2} and \mbox{SA 2864/3-1}).

\end{document}